\documentclass{amsart}

\usepackage[english]{babel}

\usepackage{amssymb,latexsym,amsmath,amsthm,graphicx,color}

\usepackage[pdftex,bookmarks,colorlinks,breaklinks]{hyperref}

\newtheorem{thm}{Theorem}
\newtheorem{lemma}{Lemma}

\newtheorem{conj}{Conjecture}

\newtheorem*{thm-others}{Theorem}

\newtheorem{innercustomthm}{Theorem}
\newenvironment{customthm}[1]
  {\renewcommand\theinnercustomthm{#1}\innercustomthm}
  {\endinnercustomthm}

\theoremstyle{remark}

\newcommand{\RR}{\mathbb{R}}

\newcommand{\ZZ}{\mathbb{Z}}

\newcommand{\emd}{\textemdash \hspace{4pt}}
\newcommand{\ol}{\overline}
\newcommand{\p}{\partial}

\title{Searching for the maximal valence of harmonic polynomials: a new example}
\author{Seung-Yeop Lee, Andres Saez}

\begin{document}

\begin{abstract}

We find a new lower bound for the maximal number of zeros to harmonic polynomials, $p(z)+\ol{q(z)}$, when $\deg p = n$ and $\deg q = n-2$.  

\end{abstract}

\maketitle

\section{Introduction and result}

Given two polynomials $p(z)$ and $q(z)$ of degrees $n$ and $m$ respectively, the maximal number of roots (i.e. maximal valence) of the {\em harmonic polynomial}, $p(z)+\overline {q(z)}$, is not known \cite{K-N2} except for a few cases (e.g. when $m=n-1$ \cite{Wilm94,Wilm98} and when $m=1$ \cite{G2008,K-S}).  See also \cite{LSL,BHJR,BHS1995,BL2004,BL2010,PSch1996}.
Recently there have been several results \cite{KLS,HLLM,LLL} on the {\em lower bounds} of the maximal valence, see Table \ref{table1}.   

\begin{table}[ht]
\caption{The known maximal valence of $p+\overline q$ \label{table1}}
\begin{tabular}{c|ccccc}
\hline\hline 
$(\deg p,\deg q)$  &  $(n,m)$   & $(n,n-1)$ & $(n,n-3)$ & $(n,1)$ 
\\\hline
\text{maximal valence} &  $\geq m^2 + m+n$  & $n^2$ & $\geq n^2-3n+{\mathcal O}(1) $
& $3n-2$ 
\\\hline
\end{tabular}
\end{table}

In this paper we suggest a new lower bound of the maximal valence when $(\deg p,\deg q)=(n,n-2)$, by studying specific harmonic polynomials defined below.  

Given a positive integer $n$ let us define two polynomials,
$p(z)= S(z)+T(z)$ and $q(z)= S(z)-T(z)$, where
\begin{equation}\label{ST} S(z)=i z^n,\qquad T(z)=i\big(z+1\big)^{n-1}\big(z-(n-1)\big).
\end{equation}
It follows that $\deg p=n$ and $\deg q=n-2$.  Since the maximal valence for $(n,m)=(3,1)$ is known (see the above table), we only consider $n\geq 4$ in this paper.

\begin{customthm}{\!\!} 
Given $n\geq 4$,
let the polynomials $p$ and $q$ be given as above.
Let $k_\text{max}(n)$ be defined by
$$ k_\text{max}(n)=\max_{1\leq k\leq n/2}\left\{k:(n-2)\cot\left(\frac{2k-1}{2n-4}\pi\right) - n \cot\left(\frac{\pi k}{n}\right)>0\right\}. $$
Then the total number of zeros, counting the multiplicity, of $p(z)+\overline{q(z)}$ is given by
$$n^2-2n+2 + 4 k_\text{max}(n). $$
The asymptotic behavior of $k_\text{max}(n)$ as $n\to\infty$ is given by
$$  k_\text{max}(n) = \left(\frac{1}{4}-\frac{X}{2\pi}\right) n  +            {\mathcal O}\left(1\right)\approx 0.13237 n + {\mathcal O}\left(1\right)  $$ where 
$X\approx 0.73908513321516$ is the unique solution to the equation $X = \cos X$. 
\end{customthm}

\begin{table}[ht]
\caption{The number of zeros, $n^2-2n+2 + 4 k_\text{max}(n)$, of $p + \ol{q}$ for small $n$'s}
\begin{tabular}{ c | c }
\hline\hline
$n$ & Number of zeros\\ \hline 
4&10\\ \hline
5&17\\ \hline
6&26\\ \hline
7&37\\\hline
8&54\\\hline
9&69\\\hline
10&86\\\hline
11&105\\\hline
12&126\\\hline
13&149\\\hline
14&174\\\hline
15&201\\\hline
16&234\\\hline
17&265\\\hline
18&298\\\hline
19&333\\\hline

\end{tabular}
\quad
\begin{tabular}{ c | c }
\hline \hline
$n$ & Number of zeros \\ \hline
20&370 \\ \hline
21&409\\\hline
22&450\\\hline
23&497\\\hline
24&542\\\hline
25&589\\\hline
26&638\\\hline
27&689\\\hline
28&742\\\hline
29&797\\\hline
30&854\\\hline
31&917\\\hline
32&978\\\hline
33&1041\\\hline
34&1106\\\hline
35&1173\\\hline

\end{tabular}
\end{table} \vspace{6pt}

\noindent{\bf Remark 1.}
For general $n$ and $m$, there exists a conjecture by Wilmshurst \cite{Wilm94} on the largest valence of the harmonic polynomials.
Though the conjecture has been disproved \cite{LLL,HLLM} for a number of cases, it has not been checked for many other cases including the case considered in this paper.
Our theorem says that the maximal valence is greater at least by $$4 k_\text{max}(n) -2 \approx 0.52948 n +{\mathcal O}(1) \text{~~~ as $n\to\infty$} $$ than the conjectured value of $n^2-2n+4$.   Our theorem also improves upon the more recent conjecture by the authors that suggests $n^2 - 3n/2 + {\mathcal O}(1)$ for the asymptotic maximal valence as $n$ grows to the infinity.  In fact, the currect project is motivated by the latter conjecture.

\bigskip
\noindent{\bf Remark 2.}
The specific harmonic polynomials that we consider in this paper are not new.  The same polynomials appeared in \cite{LLL}.  However, instead of obtaining a {\em lower bound} on the number of roots, here we obtain the {\em exact} number of roots for the given polynomials.  (If one naively applies the method in \cite{LLL} one would only get $n^2-2n+2$ as a lower bound.)   It is curious whether the similar analysis (i.e. exact counting) can be done for the case of $m=n-3$ that was considered in \cite{LLL}.

\begin{figure}[h]
\includegraphics[scale=0.35]{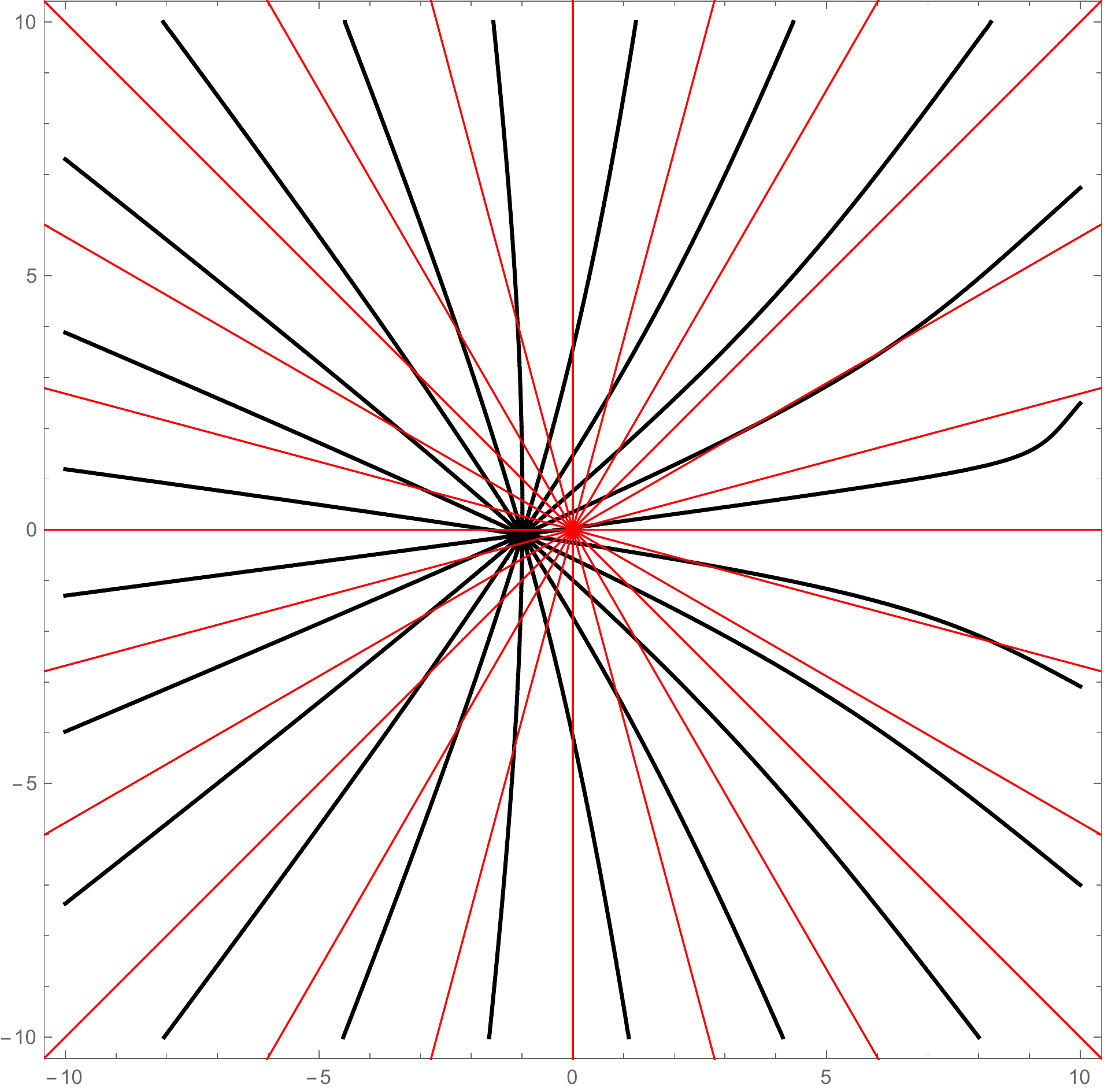}
\caption{\label{fig:2} The zero sets of $\mathrm{Im}\,T$ (black) and $\mathrm{Re}\,S$ (red) as defined in Remark 3.}
\end{figure}

\bigskip
\noindent{\bf Remark 3.} We note that our valence, $n^2-2n+2+4k_\text{max}(n)$, is not be the maximal valence.   In fact, for some $n$'s, we could find harmonic polynomials with higher valence.   The example shown in Figure \ref{fig:2} is generated by defining $p(z)=S(z) + T(z)$ and $q(z) = S(z) - T(z)$, where $S(z)=i z^{12}$ and $T(z)=i(z+e^{\frac{i}{10}})^{11}(z-11e^{\frac{i}{10}})$.  The number of zeros is $128$, two more than $126$. We conjecture that that the maximal valence is either $n^2-2n+2+4k_\text{max}(n)$ or $n^2-2n+4+4k_\text{max}(n)$ depending on $n$.

\bigskip
\noindent{\bf Remark 4.}  We conjecture that, given the degrees, $n=\deg p$ and $m=\deg q$, there exists no polynomial formula in $n$ and $m$ that gives the maximal valence of the harmonic polynomial $p+\overline q$ for all (except possibly for a {\em finite} number of cases) $n$ and $m$.
If there exists such a polynomial formula, say $P(n,m)$, then for $m=n-2$, we must have $P(n,n-2)=n^2 + A n + B$ for some constants $A$ and $B$.  The constants $A$ and $B$ must be intergers since $P(n,n-2)$ must be an integer for each $n$.   Our theorem indicates that the only possibilities are either $A=0$ or $A=-1$.  To match the known cases, $P(3,1)=7$ and $P(4,2)\in\{12,14,16\}$\footnote{For $(n,m)=(4,2)$ the maximal valence achieved so far is 12.  The maximal valence for even $n$ needs to be even due to the argument principle.  See, for example, \cite{KLS}.}, one must have $A=0$ and $B=-2$.   And this gives $P(n,n-2)=n^2-2$ which is unlikely based on the known data.

\bigskip
\noindent{\bf Acknowledgement.}  The authors would like to thank Prof. Dmitry Khavinson and Prof. Catherine B\'en\'eteau for the motivating discussions in the early stage of the project. The first author was supported by Simons Collaboration Grants for Mathematicians. 

\section{Proof}

\begin{figure}

\includegraphics[width=0.40\textwidth]{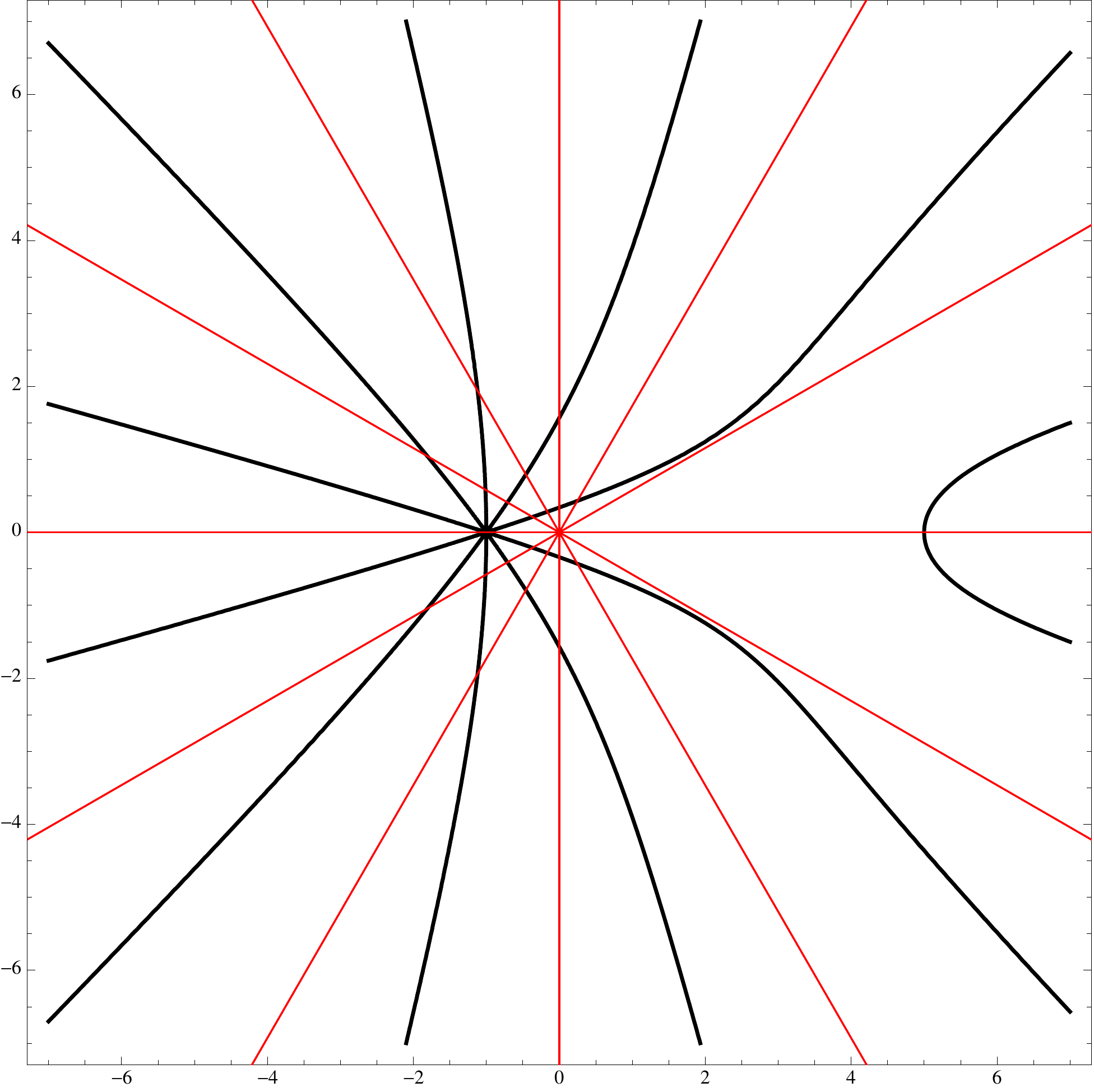}
\includegraphics[width=0.40\textwidth]{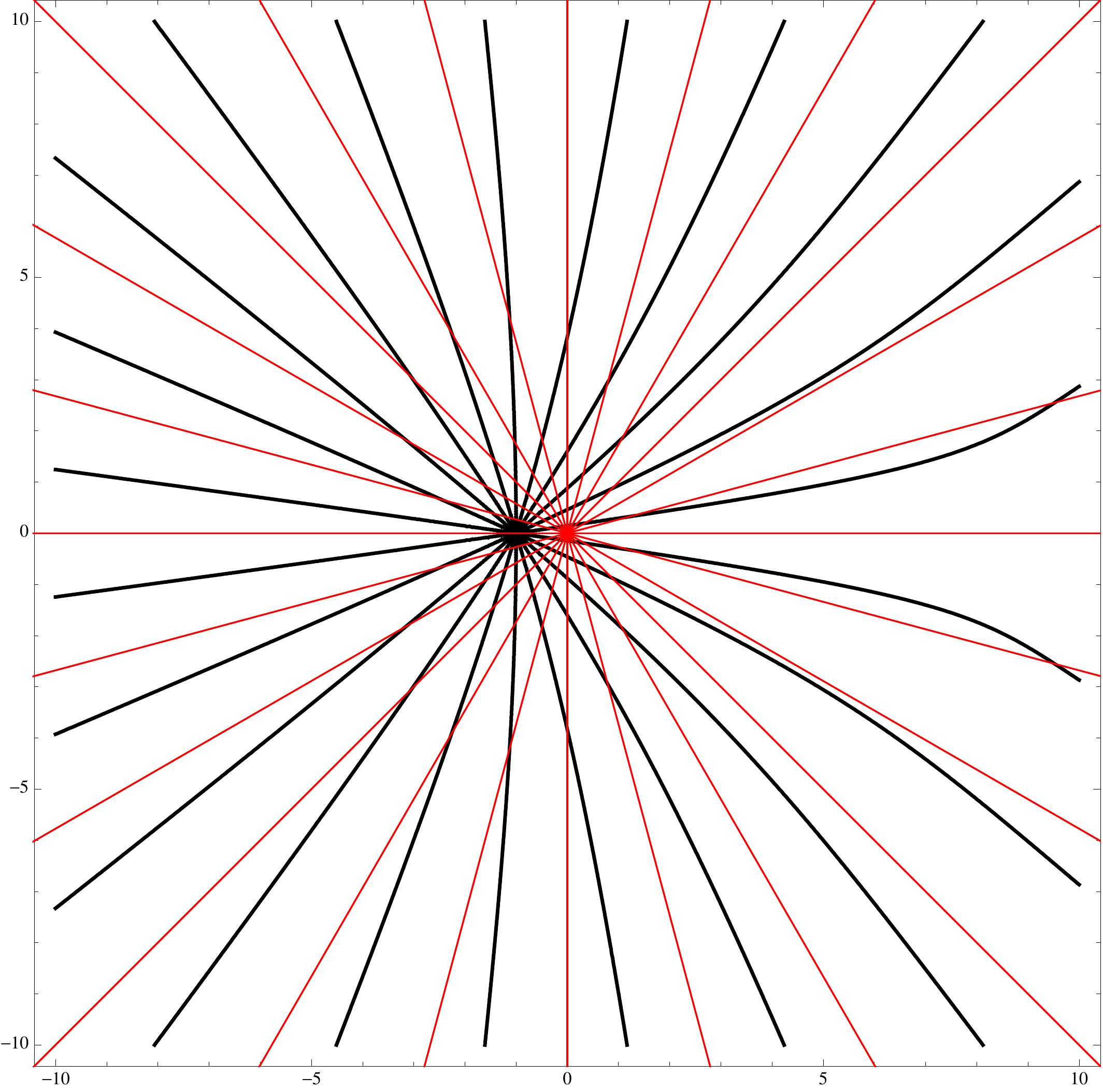}

\caption{\label{fig:1} The zero sets of ${\rm Re}\,S$ (red) and ${\rm Im}\,T$ (black) for $n=6$ and $12$ (from left)}
\end{figure}

We assume $n\geq 4$ throughout this section.

From \eqref{ST} above, we have
 $ p(z)+\overline{q(z)}=2 \,{\rm Re}\,S(z)+2 i \,{\rm Im} \,T(z)$.
Therefore, the zeros of $p(z)+\overline{q(z)}$ are exactly the intersection points of $\{z:{\rm Re} \,S(z)=0\}$ and $\{z:{\rm Im} \,T(z)=0\}$.  In Figure \ref{fig:1} the former set is depicted in the red lines and the latter set in the black curves.  The zero set of ${\rm Re}\,S$ is given explicitly by the union of $2n$ rays, i.e.,
$$ \{z:{\rm Re} \,S(z)=0\}=\bigcup_{k=-n+1}^{n}  \{r e^{i\pi k/n}:0\leq r<\infty\}.$$  
Therefore, to find the zeros of $p(z)+\overline{q(z)}$,
it is enough to find the number of intersections on each ray, i.e.
$$N_k:=\#\text{$\{r\in(0,\infty): {\rm Im}\,T\big(re^{i\pi  k/n}\big)=0\}$},$$ for each $k=-n+1,\cdots,n$. 

From the expression of $T$ in \eqref{ST}, we obtain $N_0=1$ and $N_n=n-1$ (counting the degeneracy).  We also obtain that ${\rm Im}\, T(re^{-i\pi k/n})={\rm Im}\, T(re^{i\pi k/n})$ and, therefore, $N_k = N_{-k}$.   As a result we only need to find $N_k$ for $k=1,\cdots,n-1$.

\begin{lemma} For $1\leq k\leq n-1$, $N_k$ is given by the number of zeros of $A:(0,\pi k/n)\to \RR$ where 
$$ A(\theta)=A(\theta;k)=\tan[(n-1) \theta]+ \frac{n-1}{\tan\theta}-n\cot\frac{\pi k}{n}.$$
\end{lemma}

\begin{proof}
Let us evaluate
\begin{align*}
\arg T\big(re^{i\pi  k/n}\big)=(n-1)\,\theta(r)+ \phi(r) \mod 2\pi,
\end{align*}
where we define the two (angular) variables,
\begin{equation}\label{thetaphi}
\begin{split}
\theta(r)&=\theta(r;k)=\arg\big(re^{i\pi  k/n}+1\big),
   \\
\phi(r)&=\phi(r;k)=\arg\big(re^{i\pi k/n}-(n-1)\big)+\frac{\pi}{2}. \end{split}
\end{equation}
This gives 
\begin{align}\label{tanargT}
\tan \arg T\big(re^{i\pi  k/n}\big)&=\frac{\tan\big((n-1) \theta(r)\big)+ \tan\phi(r)}{1-\tan\big((n-1)\theta(r)\big)\, \tan\phi(r)}.
\end{align}
Using the identities
\begin{equation*}
\begin{split}
\tan\theta (r)&=\tan\arg\big(re^{i\pi  k/n}+1\big)
=\frac{r \sin \left(\frac{\pi   k}{n}\right)}{r \cos
   \left(\frac{\pi  k}{n}\right)+1},
   \\
\tan\phi(r)&=\tan\left(\arg\big(re^{i\pi k/n}-(n-1)\big)+\frac{\pi}{2}\right) = -\frac{r \cos
   \left(\frac{\pi k}{n}\right)-(n-1)}{r \sin \left(\frac{\pi k}{n}\right)},
   \end{split}
\end{equation*}
one can relate $\phi$ and $\theta$ by
$ \tan\phi(r)=\frac{n-1}{\tan\theta(r)}-n\cot\frac{\pi k}{n}$.  
Using this relation, the numerator of $\tan\arg T(re^{i\pi  k/n})$ in \eqref{tanargT} is written as
$$ \tan\big((n-1) \theta(r)\big)+ \tan\phi(r) = \tan\big((n-1) \theta(r)\big)+ \frac{n-1}{\tan\theta(r)}-n\cot\frac{\pi k}{n}, $$
which is exactly $A(\theta(r))$ as defined in the lemma.
Since we have ${\rm Im}\,T(z) =0$ (note $T(z)\neq 0$ away from the real axis) if and only if $\tan\arg T(z)=0$, $N_k$ is given by the number of zeros of $A(\theta(r))$ over $r\in(0,\infty)$.  It is simple to check that the denominator in \eqref{tanargT} does not vanish when the numerator vanishes.

Lastly, from \eqref{thetaphi}, one can see that $\theta(r)$ (the angle from $-1$ to a point on the straight ray parametrized by $r$) increases from zero to $\pi k/n$ monotonically as $r$ moves from zero to $\infty$.
\end{proof}

In the rest of the proof, we will use elementary argument (e.g. the mean value theorem and the intermediate value theorem) to count the zeros of $A(\theta)$.  See Figure \ref{fig:A} for some plots of $A$.
\begin{figure}\begin{center}
\includegraphics[width=0.4\textwidth]{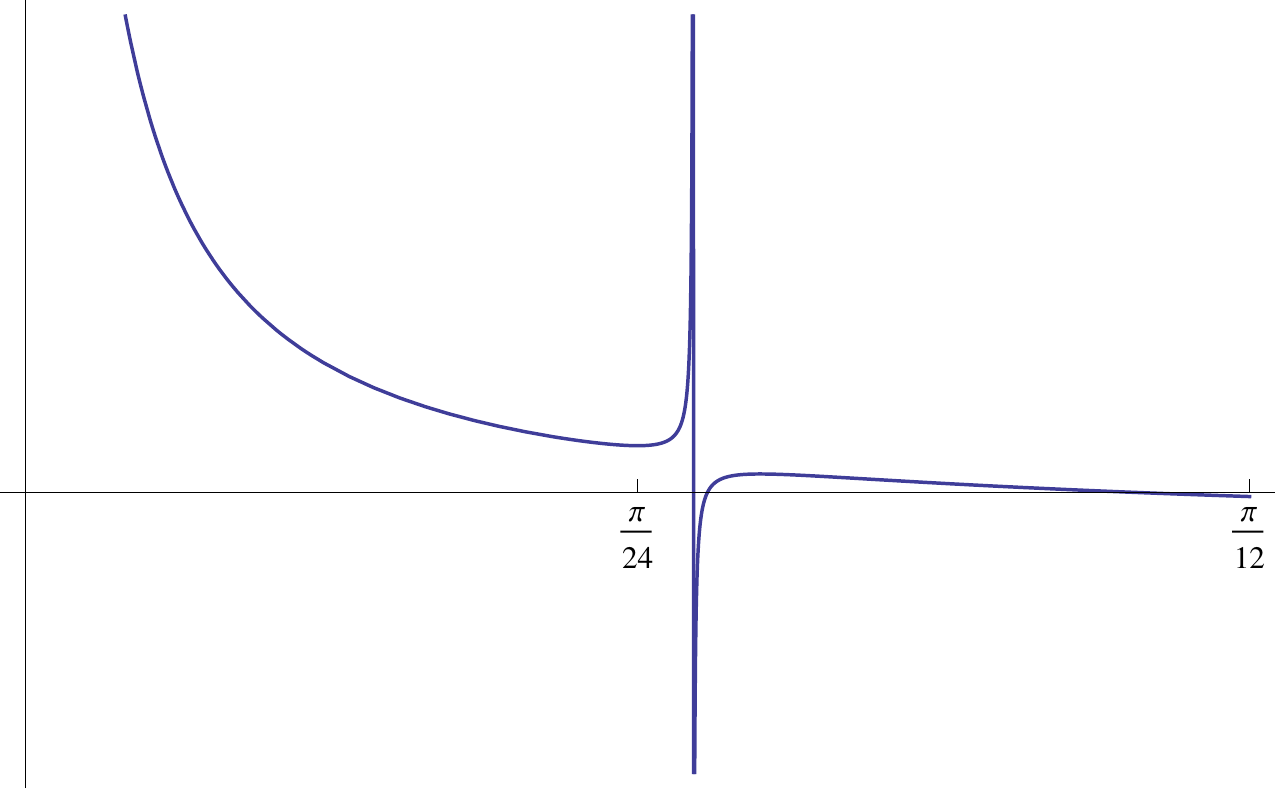}\qquad
\includegraphics[width=0.4\textwidth]{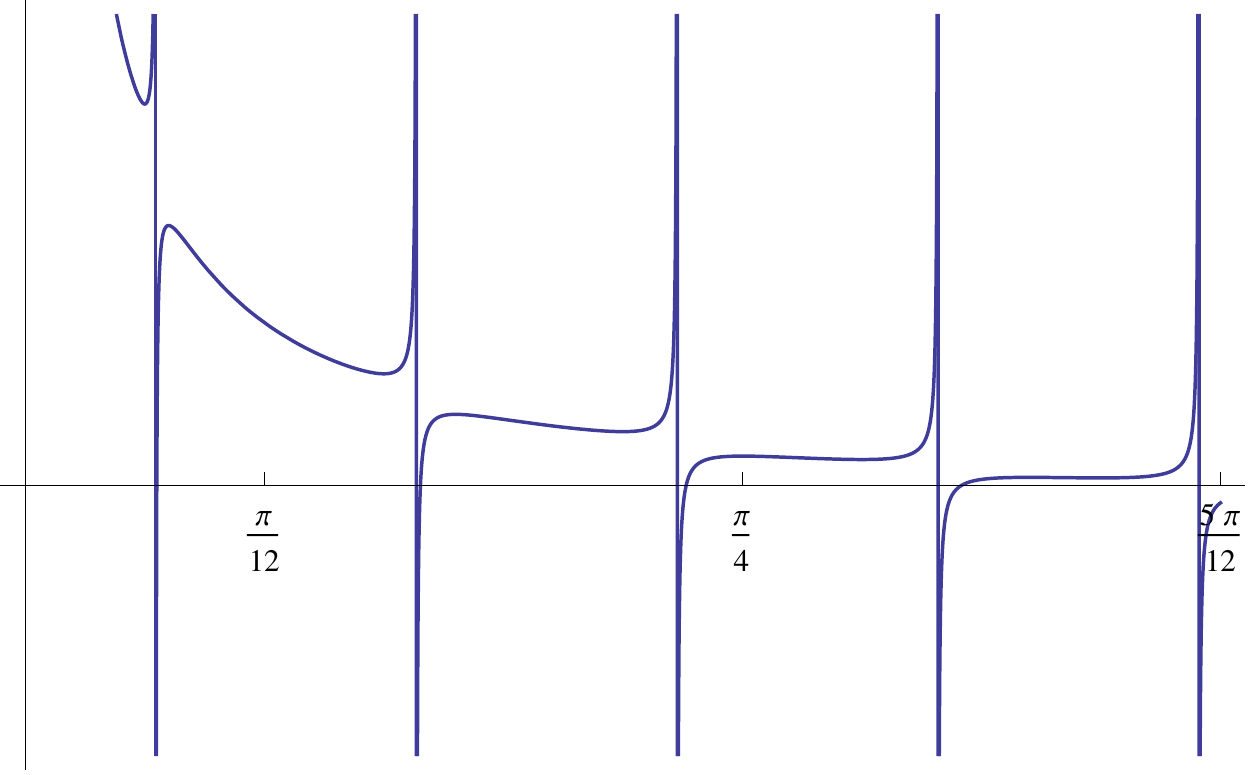}
\caption{\label{fig:A} Plots of $A(\theta)$ for $n=12$ and $k=1$ (left) and $k=5$ (right)}
\end{center}
\end{figure}

One notices that $A$ has simple poles (of negative residue) where $\tan[(n-1)\theta]$ has poles, i.e.,
$$ \theta=\frac{1/2}{n-1}\pi,\frac{3/2}{n-1}\pi,\cdots,\frac{k_\text{poles}-1/2}{n-1}\pi , $$
where $k_\text{poles}$ is the largest integer such that $\frac{k_\text{poles}-1/2}{n-1}<  \frac{k}{n}$.  We have
\begin{equation}\label{kpole}
 k_\text{poles} =k_\text{poles}(k) = \begin{cases} k\quad &\text{when~~} k < n/2, \\ k-1 \quad &\text{when~~} k \geq n/2.\end{cases} 
\end{equation}
We also get the boundary behavior of $A$: $\lim_{\theta\to 0+} A(\theta)=+\infty$ and
\begin{equation}\label{Aboundary}
A\Big(\frac{\pi k}{n}\Big) = -\tan\Big(\frac{\pi k}{n}\Big)-\cot\Big(\frac{\pi k}{n}\Big)=\displaystyle{-\frac{2}{\sin(2\pi k/n)}}=
\begin{cases}\leq 0\quad &\text{when~~} k < n/2, 
\\ >0 \quad &\text{when~~} k\geq n/2.  \end{cases}
\end{equation}
(When $k=n/2$, $\lim_{\theta\to \pi k/n} A(\theta)=+\infty$.)
To find the critical points of $A$, we evaluate
$$ A'(\theta) = (n-1) \bigg(\frac{1}{\cos^2[(n-1)\theta]}-\frac{1}{\sin^2\theta}\bigg), $$
that becomes zero when $$\cos[(n-1)\theta]=\pm \sin\theta \Longleftrightarrow  (n-1)\theta = \frac{\pi}{2}\pm\theta + \pi j ~ (\text{for~} j\in\ZZ),$$
or, equivalently, when
$$  \theta= \begin{cases}\vspace{0.2cm}  \displaystyle\frac{j-1/2}{n}\pi \text{~~ where ~~} j=1,\cdots,k, \\  \displaystyle \frac{j-1/2}{n-2}\pi  \text{~~ where ~~}  j=1,\cdots, k_\text{crit}. \end{cases} $$
We note that $k_\text{crit}$ is the largest integer such that $\frac{k_\text{crit}-1/2}{n-2} < \frac{k}{n}$ or, equivalently,
$$ k_\text{crit} =k_\text{crit}(k) = \begin{cases} k  &\text{when $k < n/4$,}\\  k-1 &\text{when $n/4\leq k < 3n/4$,} \\ k-2 &\text{when $k\geq 3n/4$.} \end{cases}  $$ 
The critical values are given by
\begin{align}\label{positive-crit}
\begin{cases}\displaystyle
A\left(\frac{j-1/2}{n-2}\pi\right)=(n-2)\cot\left(\frac{j-1/2}{n-2}\pi\right)  - n\cot \Big(\frac{\pi k}{n}\Big),
\vspace{0.2cm}\\\displaystyle
A\left(\frac{j-1/2}{n}\pi\right)=n\cot\left(\frac{j-1/2}{n}\pi\right)  - n\cot \Big(\frac{\pi k}{n}\Big)>0.
\end{cases}
\end{align}
The last inequality is from the monotonicity of $\cot(x)$ over $0<x<\pi$.

\begin{lemma}\label{lem-crit} For $1\leq k\leq n-1$ we have
$$ A\left(\frac{j-1/2}{n-2}\pi\right) >0 \text{~~ for ~~} 1\leq j\leq \min\{k_\text{crit},k-1\}.$$
\end{lemma}

\begin{proof} Since $A\big(\frac{j-1/2}{n-2}\pi\big)$ is monotonic in $j$, it is enough to prove that
$$  A\left(\frac{k'-1/2}{n-2}\pi\right) =(n-2)\cot\left(\frac{k'-1/2}{n-2}\pi\right)  - n\cot \Big(\frac{\pi k}{n}\Big)  >0, $$
for $k'=\min\{k_\text{crit},k-1\}$.
Using the following identity,
\begin{equation}\label{identity}
(n-2) \cot\theta_1 - n \cot\theta_2  
%(n-1)(\cot\theta_1-\cot\theta_2)-(\cot\theta_1+\cot\theta_2)\\
=\frac{(n-1)\sin(\theta_2-\theta_1) -\sin(\theta_1+\theta_2)}{\sin\theta_1\sin\theta_2},
\end{equation}
the above inequality in question becomes
\begin{align*} F_1(k):=(n-1)\sin \left(\frac{3n-4k}{2n^2-4n}\pi \right)-\sin \left(\frac{k-3/2}{n-2}\pi+\frac{\pi k}{n}\right)>0 \text{ when $k< 3n/4$},   
\\
F_2(k):=(n-1)\sin \left(\frac{5n-4k}{2n^2-4n}\pi \right)-\sin \left(\frac{k-5/2}{n-2}\pi+\frac{\pi k}{n}\right)>0 \text{ when $k \geq 3n/4$}. 
\end{align*}
We have $F_1(k)>0$ when
$$ k > C_n:=\frac{n (2 n-1)}{4 (n-1)} $$
because both terms in $F_1(k)$ contribute positively (we defined $C_n$ such that the second term of $F_1(k)$ vanishes when $k=C_n$).  For $k\leq C_n$, since the first term in $F_1(k)$ decreases monotonically in $k$ and is given by (using $\sin x > \frac{2x}{\pi}$ for $0<x<\pi/2$)
$$ (n-1)\sin \frac{\pi}{2(n-1)}>1  \text{ ~~when $k=C_n$}, $$
$F_1(k)>0$ for all $k \leq C_n$.

Similarly, $F_2(k)>0$ for $k\geq 3n/4$ because both terms in $F_2(k)$ contribute positively.
\end{proof}

From \eqref{positive-crit} and Lemma \ref{lem-crit}, the only possible critical point $\theta$ such that $A(\theta)\leq 0$ occurs at
$$ \theta= \frac{k-1/2}{n-2}\pi,\quad k <\frac{n}{4}. $$
Note that this is bigger than $\frac{k-1/2}{n-1}\pi$ which is the location of the rightmost pole of $A$.
Since all the other critical values are positive, it follows that, between any successive poles of $A$ (except between $\theta=0$ and $\theta=\frac{\pi}{2(n-1)}$), there is exactly one root of $A$, by applying the intermediate value theorem and the mean value theorem.  Therefore, there are total $k_\text{poles}-1$ zeros of $A$ that are located to the left of the rightmost pole.   By applying the intermediate value theorem with \eqref{Aboundary}, for all $k\geq n/2$, there is at least one root of $A$ to the right of the rightmost pole of $A$.  
Combining with \eqref{kpole} there are at least $k-1$ roots of $A$ for {\em all} $1\leq k\leq n-1$.
 The total number of zeros of $p+\overline{q}$ is therefore at least 
\begin{equation}\label{small} n + 2\sum_{k=1}^{n-1} (k-1) =n^2-2n+2.\end{equation}
For $k< n/2$, again applying the intermediate value and the mean value theorem between the rightmost pole and $k\pi/n$, there are
\begin{equation*} \begin{cases}\displaystyle \text{no zero when $k\geq n/4$ or when $k<n/4$ and } A\Big(\frac{k-1/2}{n-2}\pi\Big)<0,  \\\displaystyle \text{two zeros when $k < n/4$ and } A\Big(\frac{k-1/2}{n-2}\pi\Big)>0.  \end{cases} \end{equation*}
Note that the second inequality is exactly the condition to define $k_\text{max}(n)$ in the main theorem.
Using the identity \eqref{identity}, the inequality, $A\big(\frac{k-1/2}{n-2}\pi\big)>0$, holds if and only if
$$(n-1)\sin \left(\frac{n-4k}{2n^2-4n}\pi \right)-\sin \left(\frac{k-1/2}{n-2}\pi+\frac{\pi k}{n}\right)>0. $$
For $k<n/4$, the left hand side is monotonically decreasing in $k$ and, therefore, the inequality is satisfied exactly for $k\leq k_\text{max}(n)$.  Therefore we get $4 k_\text{max}(n)$ more roots of $p+\overline{q}$ than \eqref{small}.  This proves our theorem except the statement about the asymptotic behavior.

In terms of the new parameter, $\gamma= k/n$, the left hand side of the above inequality becomes
\begin{align*}
(n-1)\sin \left(\frac{1-4\gamma}{2n-4}\pi \right)-\sin \left(\frac{n\gamma-1/2}{n-2}\pi+\gamma\pi\right)
=\frac{\pi}{2}-2\pi\gamma-\sin(2\pi\gamma) + {\mathcal O}\left(\frac{1}{n}\right),
%\pi \frac{(1-4\gamma)(\cos (2\pi\gamma)+1)}{2n} +{\mathcal O}\left(\frac{1}{n^2}\right)
\end{align*}
as $n\to\infty$.
The leading term in the above expression is positive when 
 $$\frac{k}{n}=\gamma < \frac{1}{4}-\frac{X}{2\pi}  % + \frac{X}{2\pi n} + {\mathcal O}\left(\frac{1}{n^2}\right) $$
% +{\mathcal O}\left(\frac{1}{n}\right)
 $$
where $X\approx 0.739$ is the unique solution to the equation $X=\cos X$.  It implies that $\frac{k_\text{max}}{n} = \frac{1}{4}-\frac{X}{2\pi}+{\mathcal O}\left(\frac{1}{n}\right)$.

\end{document}